\documentclass[12pt]{amsart}
\usepackage{latexsym, amsmath, amssymb, amscd}
\usepackage[colorlinks]{hyperref}

\theoremstyle{plain}
\newtheorem{Thm}{Theorem}
\newtheorem*{Thm*}{Theorem}
\newtheorem{Cor}[Thm]{Corollary}
\newtheorem*{Cor*}{Corollary}
\newtheorem{Prop}[Thm]{Proposition}
\newtheorem{Lma}[Thm]{Lemma}

\theoremstyle{definition}
\newtheorem{Def}[Thm]{Definition}

\theoremstyle{remark}
\newtheorem{Rem}[Thm]{Remark}

\numberwithin{equation}{section}
\newcommand{\brq}{^{[q]}}

 \newcommand{\op}[1]{\operatorname{#1}}

\newcommand{\mlabel}[1]%
  {\mbox{}\marginpar{\raggedleft\hspace{0pt}{\rm\ttfamily#1}}\label{#1}}

\newcommand{\into}{\operatorname{\hookrightarrow}}

\newcommand{\e}{\operatorname{e}}

\newcommand{\length}{\operatorname{\lambda}}

\newcommand{\Ass}{\operatorname{Ass}}
\newcommand{\Assh}{\operatorname{Assh}}

\newcommand{\Hom}{\operatorname{Hom}}

\newcommand{\fm}{{\mathfrak m}}
\newcommand{\m}{\fm}
\newcommand{\n}{{\mathfrak n}}

\newcommand{\ringR}{\text{$(R,\fm,k)$ }}

\newcommand{\et}{0^{\ast}_{E_{R}}}

\newcommand{\bx}{\bold x}

\newcommand{\inc}{\subseteq}
\newcommand{\ehk}{\e_{HK}}

\newcommand{\8}{\infty}

\newcounter{hours}\newcounter{minutes}

\numberwithin{Thm}{section}
\numberwithin{equation}{section}

\newcommand{\excise}[1]{}

\begin{document}
\title
{Lower Bounds for
Hilbert-Kunz multiplicities in local rings of fixed dimension}

%\texttt{{--Working  version--}}

\author{Ian M. Aberbach}

\address{Department of Mathematics, University of Missouri, Columbia, MO 65211}\email{aberbach@math.missouri.edu}

\author{Florian Enescu}

\address{Department of Mathematics and Statistics, Georgia State University, Atlanta, GA 30303}\email{fenescu@gsu.edu}
\thanks{The second author was partially supported by an Young Investigator 
Grant H98230-07-1-0034 from the National Security Agency.}
\maketitle

\begin{abstract} Let $(R,\m)$ be a formally unmixed local ring of positive prime
characteristic and 
dimension $d$.   We examine the implications of having
small Hilbert-Kunz multiplicity (i.e., close to $1$).
In particular, we show that if $R$ is not regular, there
exists a lower bound, strictly greater than one, depending only on $d$,
 for its Hilbert-Kunz multiplicity.
\end{abstract}

\section{Introduction}\label{intro}

Let $\ringR$ be a local ring of positive characteristic $p$, that is,
quasi-local (only one maximal ideal) and Noetherian. Let $q
=p^e$, where $e$ is a nonnegative integer. For any ideal $I$ of $R$ we
denote $I^{[q]} = ( i^q : i \in I).$

For an $\fm$-primary ideal $I$, one can consider the Hilbert-Samuel multiplicity and
the Hilbert-Kunz multiplicity of $I$ with respect to $R$.

\begin{Def}
 {\rm Let $I$ be an $\fm$-primary ideal in $(R,\fm)$. 
Let $\length(-)$ denote the usual length function.

1. {\it The Hilbert-Samuel multiplicity of $R$ at $I$} is defined by $\e (I)=
\e(I;R) := \displaystyle\lim_{n \to \infty} d!\, \frac{\length(R/I^n)}{n^d}$. The
limit exists and it is a positive integer.

2. {\it The Hilbert-Kunz multiplicity of $R$ at $I$} is defined by $\e _{HK}
(I)= \e _{HK}(I;R): = \displaystyle\lim_{q \to \infty} 
\frac{\length(R/I^{[q]})}{q^d}$.  Monsky has shown that the latter
 limit exists and 
is positive.} 

The Hilbert-Samuel multiplicty of $R$, denoted $\e(R)$, is by definition
$\e(\m)$. Similarly,
the Hilbert-Kunz multiplicity of $R$, denoted  $\ehk(R)$, is $\ehk(\m)$.
\end{Def}

It is known that for parameter ideals $I$, one has $\e(I) = \e_{HK}(I)$. The
following sequence of inequalities is also known to hold whenever
$I$ is $\m$-primary: 
$${\rm max} \{ 1,
\dfrac{\e(I)}{d!}\} \leq \e_{HK} (I) \leq \e(I).$$ 

We call a local ring $R$ {\it formally unmixed} if  
$\hat{R}$ is equidimensional and ${\rm Min}(\hat{R}) =
{\rm Ass}(\hat{R})$, that is, 
$\dim(\hat{R}/P) = \dim(\hat{R})$ for all its minimal primes $P$, and
all associated primes of $\hat{R}$ are minimal. Nagata calls such
rings {\it unmixed}. 
However, throughout our paper, a local
unmixed ring is a local ring $R$ that is equidimensional and ${\rm Min}(R) =
{\rm Ass}(R)$. 

% A formally
% unmixed local ring is a local ring $\ringR$ such that its $\m$-adic completion
% $\hat{R}$ is unmixed.

In this paper we investigate  rings that have small Hilbert-Kunz
multiplicity. It is known that a formally unmixed local ring of characteristic $p$ is
regular if and only if $\e_{HK}(R) =1$. 
 In fact, similar statements hold true for the Hilbert-Samuel
multiplicity and they are considered classical.
(The unmixedness assumption is
essential as there are examples of nonregular rings that are not
formally unmixed with
$\e_{HK}(R) =1$.  The reason is that neither Hilbert-Samuel multiplicity
nor Hilbert-Kunz multiplicity can pick up lower dimensional components
of $\hat{R}$). Since $\e(R)$ is always a
positive integer we have that $\e(R) \geq 2$ if $R$ is formally unmixed 
but not regular. The
situation is much more subtle in the case of the Hilbert-Kunz multiplicity
because it often takes on non-integer values. So, the question becomes: If
one fixes the dimension $d$, how close to $1$ can $\e_{HK}(R)$ be (when $R$ is
formally unmixed, but not regular)? What can be said about the structure of rings of small
Hilbert-Kunz multiplicity? This problem has been intensively
 studied in recent years (with success mostly
for rings of small dimension) by Blickle-Enescu~\cite{BE}, 
Watanabe-Yoshida~\cite{WY}, \cite{WY2}, \cite{WY3}, 
and Enescu-Shimomoto~\cite{ES}.  In the current paper, we will develop 
techniques that shed light on this problem independent of dimension. We show that if $R$ is not regular, there
exists a lower bound, strictly greater than one, depending only on $d$,
 for its Hilbert-Kunz multiplicity.

The goal is at least twofold: find the following constants (as introduced
in~\cite{BE}),

\[
    \epsilon_{HK}(d,p) = \inf\{\e_{HK}(R)-1 : \text{$R$ non-regular,
    formally unmixed, $\dim R = d$, $\op{char} R = p$} \}
\]
\noindent 
and 

\[ \epsilon_{HK}(d) = \inf\{\epsilon_{HK}(d,p) : p > 0\} \] \noindent and
describe the structure of the rings with small Hilbert-Kunz multiplicity from
both an algebraic and geometric point of view.  

It is known that $\epsilon_{HK}(d,p) \geq \dfrac 1{d!p^d}$ by results in~\cite{BE}. 
Clearly, however, as $p\to \8$, the right hand side tends toward $0$, so
this does not give a positive lower bound for $\epsilon_{HK}(d)$.  A
byproduct of our work is that it leads us to a proof of the fact that
$\epsilon_{HK}(d) >0$, answering positively a problem raised in~\cite{BE},
Section 3. We should mention  that a conjecture of Watanabe and
Yoshida~\cite{WY3} asserts that if $\ringR$ has residue field equal to
$\overline{\mathbf{F}_p},\, p>2$, then $\ehk(R) \geq \ehk(R_{p,d})$, where
$R_{p,d} = \overline{\mathbf{F}_p}[[x_0,\ldots, x_d]]/(x_0^2+\cdots+x_d^2).$
This conjecture has been answered positively for dimensions $d=1,2,3,4$ (the
difficult cases of dimension $3,4$ are due to Watanabe and Yoshida) and in the
case of complete intersections by Enescu and Shimomoto (\cite{ES}).

The starting point of our investigation is the following:
\begin{Thm}[Blickle-Enescu]\label{thm:BE}
\label{be} Let $R$ be an unmixed $d$- dimensional ring that is a homomorphic image
of a  Cohen-Macaulay  local ring  of characteristic  $p  > 0$. Let $d \geq 2$. If
$$\e_{HK}(R) \leq 1 + \op{max}\{1/d!,1/\e(R)\},$$ then $R$ is
Cohen-Macaulay and F-rational.
\end{Thm}

\begin{Rem}
The proof of the above result shows that, in 
fact, the inequality $\ehk(R) < \dfrac{\e(R)}{\e(R)-1}$ forces $R$
to be Cohen-Macaulay and F-rational.
\end{Rem}

In fact, the hypotheses of Theorem~\ref{thm:BE} suffice to show that 
$R$ must be (strongly) F-regular.  This is the content of 
Corollary~\ref{smallehk} which states:
\begin{Cor*}
 Let
$(R,\fm,k)$ be a formally unmixed ring  of characteristic $p$ and $\dim(R) = d \geq 2$.
  If
$\e_{HK}(R) \leq 1 + \op{max}\{1/d!,1/\e(R)\}$, then $R$ is 
F-regular and Gorenstein. 
If $R$ is excellent, then $R$ is strongly F-regular. 
\end{Cor*}

Theorem~\ref{lowerbound} gives a positive lower bound for $\epsilon(d)$
which does not depend on $p$:
\begin{Thm*}
Let $\ringR$ be a formally unmixed local ring of positive characteristic $p$
and dimension $d$.
If $R$ is not regular then 
$$\ehk(R) \geq 1+ \frac{1}{d \cdot (d! (d-1)+1)^d}.$$
\end{Thm*}

While this result shows that $\epsilon(d) > 0$, our techniques can be refined to give sharper estimates. In a future paper, 
we will give results that are considerably better, but the cost is that the arguments are very much
more technical, so we have opted to give a more accessible proof of the fact that such an $\epsilon(d)$ exists. Although the above mentioned conjecture of Watanabe and Yoshida is still open, 
we have developed techniques that, for the first time, work regardless of dimension or additional hypotheses on the rings.

In dealing with Hilbert-Kunz multiplicities it  often useful to assume that 
the
rings that are studied are either formally unmixed or unmixed and homomorphic
images of Cohen-Macaulay rings. This will also be the case in our paper.

\medskip
\noindent
{\bf Acknowledgment:} We thank the referee for his/her careful reading of the original 
manuscript, and for a number of important corrections and improvements.

\section{Definitions and known results}\label{known}
First we would like to review some definitions and results that will be useful
later. Throughout the paper $R$ will be a Noetherian ring containing a field of 
characteristic $p$, where $p$ is prime.
Also, $q$ will denote $p^e$, a varying power of $p$.

If $I$ is an ideal in $R$, then $I^{[q]}=(i^q: i \in I)$, where
$q=p^e$ is a power of the characteristic. Let $R^{\circ} = R \setminus
\cup P$, where $P$ runs over the set of all minimal primes of $R$. An
element $x$ is said to belong to the {\it tight closure} of the ideal
$I$ if there exists $c \in R^{\circ}$ such that
$cx^q \in I^{[q]}$ for all sufficiently large $q=p^e$. The
tight closure of $I$ is denoted by $I^\ast$. By a ${\it parameter \
ideal}$ we mean here an ideal generated by a full system of parameters
in a local ring $R$. A tightly closed ideal of $R$ is an ideal $I$
such that $I = I^*$.

Let $F:R \to R$ be the Frobenius homomorphism $F(r)=r^p$. We denote by
$F^e$ the $e$th iteration of $F$, that is $F^e(r) = r^{q}$, $F^e:R
\to R$. One can regard $R$ as an $R$-algebra via the homomorphism
$F^e$. Although as an abelian group it equals $R$, it has a different
scalar multiplication. We will denote this new algebra by $R^{(e)}$.
For an $R$-module $M$ we let $F^e(M) = R^{(e)}\otimes_R M$, where
we consider this an $R$-module via $R^{(e)}$, i.e., $a(r\otimes m) = (ar)
\otimes m$, but $r \otimes (am) = a^qr \otimes m$.  
For an element $m \in M$, let $m^q = 1 \otimes m \in F^e(M)$.
If $N \inc M$ then
we denote the image of $F^e(N)$ in $F^e(M)$ by $N\brq$, and this is
the same as the submodule of $F^e(M)$ generated by the elements $n^q$ for
$n \in N$.   We then say that $x\in M$ is in the tight closure of 
$N$ in $M$, denoted $N_M^*$, if there exists $c \in R^0$ such that
$cx^q \in N\brq$ for all $q \gg 0$. 

\begin{Def}
 $R$ is \emph{F-finite}  if $R^{(1)}$ is module finite over $R$, or, 
equivalently (in the case that $R$ is reduced),
 $R^{1/p}$ is module finite over $R$.  $R$ is called 
\emph{F-pure} if  the Frobenius homomorphism is a pure map, i.e,  
 $F \otimes_R M$ is injective for every $R$-module $M$.
\end{Def}

If $R$ is F-finite, then $R^{1/q}$ is module finite over $R$, for
every $q$. Moreover, any quotient and localization of an F-finite
ring is F-finite. Any finitely generated algebra over a perfect
field is F-finite. An F-finite ring is excellent.

\begin{Def}
A reduced Noetherian F-finite ring $R$ is \emph{strongly F-regular}  if 
for every $c \in R^0$ there exists $q$ such that the $R$-linear map $R \to R^{1/q}$ that sends $1$ to $c^{1/q}$ 
splits over $R$, or equivalently $Rc^{1/q} \subset R^{1/q}$ splits over $R$.
\end{Def}

The notion of strong F-regularity localizes well, and all ideals are tightly
closed in strongly F-regular rings. Regular  rings are strongly F-regular
and  strongly F-regular rings are Cohen-Macaulay and normal.

Let $E_R(k)$ be the injective hull of the residue field of $R$. Then an
F-finite ring reduced $R$ is
strongly F-regular if and only if $\et =0$, see for example~\cite{S1}, 7.1.2. More generally, when 
$(R,\m)$ is reduced, excellent (but not necessarily F-finite) we will say
that $R$ is strongly F-regular if $\et = 0$. 

\begin{Def} A ring $R$ is called F-rational if all parameter ideals are
tightly closed. A ring $R$ is called weakly F-regular if all ideals are 
tightly
closed. The ring $R$ is F-regular if and only if $S^{-1}R$ is weakly
F-regular for all multiplicative sets $S \subset R$. 
\end{Def}

Regular rings are (strongly) F-regular.  For Gorenstein rings, the notions
of F-rationality and F-regularity coincide (and if in addition the
ring is excellent, these coincide with strong F-regularity).

\begin{Def} Let $ I \subseteq J$ be two $\m$-primary ideals in $\ringR$ and $M$
a finitely generated $R$-module. 
The Hilbert-Kunz multiplicity of $I$ on $M$ is $\ehk(I;M) = \lim_{q\to\8}
\dfrac 1 {q^d} \length(M/I\brq M)$.  
The relative Hilbert-Kunz multiplicity of $I$
and $J$ on $M$ is 
$ \ehk(I,J; M) = \ehk(I; M) -\ehk(J; M)$.

When $M=R$, we simply drop it from the notation.

\end{Def}

\begin{Prop}[Associativity formula, see Prop 1.2 (5) in~\cite{WY3}]
\label{assoc}
Let $\ringR$ be a local ring and $I$ an $\fm$-primary ideal of $R$. Denote ${\rm Assh}(R) = \{ P \in {\rm Ass}(R): \dim(R/P) =\dim (R) \}$. Then
$$ \ehk(I;M) = \sum _{P \in {\rm Assh}(R)} \length_{R_P}{(M_P)} \cdot \ehk(I;R/P).$$  
\end{Prop}

\begin{Rem}
The associativity formula immediately implies that if $\ehk(R) < 2$ then
${\rm Assh}(R)$ contains one element, and if this is the prime $P$ then
the $P$-primary component of $0$ is $P$.  Thus, if $R$ is unmixed and
$\ehk(R) < 2$ then $R$ is a domain.
\end{Rem}
We will also need the following technical notion:

\begin{Def}  Let $\ringR$ be a local ring of positive characteristic
$p$ and let $J \subset I$ be $\fm$-primary ideals. Define 
the \emph{star length} of $J$ in $I$,
$\length^*(I/J)$, to
be the minimum length $n$ of a sequence of ideals $$J^* = I_0 \subset I_1
\subset ... \subset I_n=I^*$$ such that, for each $k$, $I_{k+1} = (I_k,x_k)^*$
for some element $x_k$ with $\fm x_k \subset I_k$. 
\end{Def}

The  definition of star length was introduced by Hanes \cite{Ha2}, 
who also noted some of the basic
properties of the star length function:

\begin{Prop} \label{star} Let $J \subset I$ be any $\fm$-primary ideals of a
local ring $\ringR$ of prime characteristic $p >0$. Then
\begin{itemize}
\item [a)] $\length^*(I/J) \leq \length(I/J)$ and $\length^*(I/J) = \length^*(I^*/J^*)$;
\item [b)] $\e _{HK}(J) \leq \e_{HK}(I) + \length^*(I/J) \e_{HK}(R)$. 
Moreover, 
$\e_{HK}(J) \leq \length^*(R/J) \e_{HK}(R)$.
\end{itemize}
\end{Prop}

The following Proposition offers a natural  characterization of strong
F-regularity in terms of the relative Hilbert-Kunz multiplicity.

\begin{Prop}\label{infehk}
Let $(R,\m,k)$ be an excellent local ring.  Then the following are equivalent:
\begin{itemize}
\item[1)] $R$ is strongly F-regular.
\item[2)] $\inf\{\e_{HK}(I,J) | I \subsetneq J\} > 0$.
\item[3)] $\inf\{\e_{HK}(I, (I,x)) | I \textrm{ is $\m$-primary, 
irreducible and $x$ is a 
socle element modulo $I$} \} > 0$
\end{itemize}
\end{Prop}
\begin{proof}
By \cite{AL}, Theorem 0.2, $R$ is strongly F-regular if and only if
$\liminf \length(R/ 0:_{F^e(E)} u^q)/q^d > 0$ (the theorem is stated
there for F-finite rings, but the proof works in the excellent case too).  

We first show that (1) implies (3).  Let $I \inc R$ be
irreducible and $\m$-primary.  Say $x$ is a socle element modulo $I$.
There is then an injection $R/I \hookrightarrow E$ sending $x$ to $u$.
Applying Frobenius gives a map $R/I\brq \to F^e(E)$ sending $x^q$
to $u^q$, from which it is clear that $I\brq:x^q \inc 0:_{F^e(E)} u^q$.
Hence $\e_{HK}(I, (I,x)) \ge \liminf \length(R/ 0:_{F^e(E)} u^q)/q^d > 0$.

To see that (3) implies (2) we note that it suffices to take
$J = (I,y)$ for a socle element $y$ modulo $I$.  In this case we
can embed $R/I \hookrightarrow R/I_1 \oplus \cdots R/I_t$ where
each $I_n$ is irreducible, and $y \mapsto (x,0,\ldots, 0)$ where
$x$ is the socle element modulo $I_1$.  It is then clear, after
applying Frobenius, that $\e_{HK}(I,J) \ge \e_{HK}(I_1, (I_1,x))$.

Clearly (2) implies (3).

  Suppose that  (3) holds,
but $R$, of dimension $d$, is not strongly F-regular.  Choose $c \in R^0$
such that $cu^q = 0$ in $F^e(E)$ for all $q$.  Then $\dim R/cR = d-1$.  Let 
$\e_1 = \e_{HK}(R)$ and $\e_2 = \e_{HK}(R/cR)$.  Fix $q_0$ such that
$\length(R/(c,\m^{[q_0]})) \le (\e_2+1)q_0^{d-1}$.  Since $cu^{q_0} =0$, we can
choose an irreducible ideal $I$ with socle representative $x$ such that
$cx^{q_0} \in I^{[q_0]}$.  Since $\m x \inc I$ we see that for all $q$,
$(\m^{[q_0]},c)\brq x^{q_0 q} \inc I^{[q_0 q]}$. Hence for large $q$ 
\begin{equation*}
\length\left(\dfrac R {I^{[q_0 q]}:x^{q_0 q}} \right) \le \length\left(\dfrac R
{(\m^{[q_0]},c)\brq }\right) \le \length\left( \dfrac R{(\m^{[q_0]},c)} \right)
(\e_1+1) q^d \le (\e_2+1) q_0^{d-1} (\e_1+1) q^d. 
\end{equation*} 
Dividing by $(q_0 q)^d$
and taking limits shows that  $\e_{HK}(I, (I,x)) \le \dfrac{(\e_2+1)(\e_1+1)}{q_0}$. 
Since $q_0$ may be taken arbitrarily large (this will change the ideal $I$), we
have contradicted the assumption (3). 
\end{proof}
In later sections we will often want to be able to obtain a minimal reduction
of an ideal in a local ring.  The standard technique is to pass to a faithfully
flat extension.  The next remark merely summarizes several well-known
facts that we will need.
\begin{Rem}\label{stdrmk}
Let $(R,\m,k)$ be a local ring of characteristic $p$.
\begin{itemize}
\item[a)] Assume that $(R,\m) \to (S,\n)$ is a flat local homomorphism
with $\n = \m S$ (e.g., completion).
\begin{itemize}
\item[i)] For any $\m$-primary ideal $I \inc R$, $\ehk(IS) = \ehk(I)$.
In particular, $\ehk(S) = \ehk(R)$.
\item[ii)] If $R$ is CM with canonical module $\omega_R$ then 
$S$ is CM with canonical module $\omega_S = \omega_R \otimes S$.  
\end{itemize}
\item[b)]  Let $Y$ be an indeterminate over $R$ and set $S = R[Y]_{\m R[Y]}$.
Then $S$ is faithfully flat with maximal ideal extended from $R$, and
residue field isomorphic to $k(Y)$ (so infinite).  Part (a) then applies.
\item[c)] If $R$ has infinite residue field then $\m$ has a minimal reduction
$\bx = x_1,\ldots, x_d$ with $\e(R) = \e((\bx)) = \ehk((\bx))$, and if
$R$ is CM then the common value is also equal to $\length(R/(\bx))$.
If $R$ has finite residue field then parts (a) and (b) may be applied
in order to change to the situation that the residue field is infinite.
\end{itemize}
\end{Rem}

\section{Hilbert-Kunz lower bounds via duality}\label{duality-bounds}

This section will present various lower bounds for the Hilbert-Kunz multplicity
of a ring $\ringR$ of fixed multiplicity and dimension. 
% We will frequently use
% the following notation: $\ehk= \ehk(R), \e= \e(R)$, $d =\dim(R)$.

We observe the following:
\begin{Lma}\label{cbound}
If $(R,\m)$ is local of dimension $d$, $I \inc J$ are
$\m$-primary ideals,  $c \in R^{\circ}$, and $M$ is finitely generated over $R$,
then 
$\lim\limits_{q\to\8}\dfrac 1{q^d}\length\left(\dfrac{J\brq M}{(cJ\brq +I\brq)M}\right)
= 0$.
\end{Lma}

\begin{proof} Let $n= \mu(M)$  and $k = \mu(J)$. Then one can see that
there is a surjection 
$$\left(\frac{R}{cR}\right)^{nk}
\to \dfrac{J\brq M}{(cJ\brq +I\brq)M} \to 0,$$ 
 and the kernel contains
$I^{[q]}\left(\frac{R}{cR}\right)^{nk}$.

Since $\dim R/cR=d-1$, 
we note that 
$\lim\limits_{q\to\8}
\dfrac 1{q^d}\length\left(\left(\dfrac{R}{cR+I^{[q]}}\right)^{nk}\right)= 0$,
which implies our statement.
\end{proof}

We are now ready to formulate an important technical result that will 
lead to a series of Corollaries which are the main goal of this section.
% We remind the reader that a formally generically Gorenstein is a ring $R$ such that with its completion is Gorenstein when localizing at each of the minimal primes
% of the completion. 

\begin{Thm}\label{ehk-hom}
Let $(R, \fm)$ be a  
Cohen-Macaulay ring with system of parameters $\bx
=x_1,\ldots, x_d$. Let $e = \length(R/(\bx))$.  
Suppose that $I \supseteq (\bx)$ and set $J = (\bx)^*:I$.

Let  $a = \length^*(R/I)$,   $f = \length^*(R/J)$, and 
$b = \length(((\bx)^*:I)/(\bx))$.  
Then $\ehk(R) \geq \dfrac{e}{f+a}$, so, in particular,
$$\ehk (R) \geq \frac{e}{e -b +a}.$$
\end{Thm}

\begin{proof}
Completing $R$ leaves the Hilbert-Kunz multiplicty unaffected, can only
increase the star lengths ($a$ and $f$), and decrease $b$.  So to prove
the desired formulas we may complete.  Hence we may assume that $R$  
has a $q_0$-weak test element $c$.

Let $\omega_R$ be the canonical module of $R$. 
We have $\Assh(R) = \Ass(R)$ and for each $P \in \Ass(R)$, 
$\length_{R_P}(\omega_P) = \length_{R_P}(R_P)$.  Hence, applying 
the associativity formula in Remark~\ref{assoc} to compute
$\ehk(I;R)$ and $\ehk(I;\omega)$, we see that they are equal.
 Hence
$\ehk (I_1, I_2; \omega_R) = \ehk (I_1, I_2)$
whenever  $I_1 \subseteq I_2$  are $\fm$-primary
ideals.

Since $\bx$ is a s.o.p.,  $\ehk((\bx)) = \e((\bx)) = e$. Also,
 $\ehk ((\bx)) = \ehk (J) + \ehk((\bx),J)$.

By Proposition~\ref{star}, 
$\ehk(J) \leq \length^*(R/J) \ehk(R) = f
\ehk(R)$.

The heart of the proof is seeing that
 $\ehk ((\bx), J; \omega_R) \leq a\ehk(R)$, and hence
$\ehk((\bx), J) = \ehk ((\bx), J; \omega_R) \le  a\ehk(R)$. 

Indeed, $\omega_R/(\bx)\brq \omega_R$ is the canonical module 
of the Artinian ring
$R/(\bx)\brq$, so it is injective over it.
 By Matlis duality over complete Artinian
rings, we get that $\length (R/I\brq) = \length \left(\Hom (R/I \brq,
\omega_R/(\bx)\brq \omega_R)\right)$.

Note that by the definition of $J$, and the fact that $c$ is a $q_0$-weak
test element, we have $c J\brq \inc (\bx)\brq:I\brq$ for all $q \ge q_0$.
Thus for all $q \ge q_0$
\begin{equation*}
\dfrac{(cJ\brq+(\bx)\brq)\omega_R}{(\bx)\brq \omega_R} \inc
\dfrac{(\bx)\brq \omega_R:I\brq}{(\bx)\brq \omega_R} 
= \Hom\left(\dfrac R{I\brq}, \dfrac{\omega_R}{(\bx)\brq \omega_R}\right).
\end{equation*}
By the equality
\begin{equation*}
\length\left(\dfrac{J\brq \omega_R}{(\bx)\brq \omega_R}\right)
= \length\left(\dfrac{J\brq \omega_R}{(cJ\brq+(\bx)\brq)\omega_R}\right) +
\length\left(\dfrac{(cJ\brq+(\bx)\brq)\omega_R}{(\bx)\brq \omega_R}\right),
\end{equation*}
Lemma~\ref{cbound},  Matlis duality, and Proposition~\ref{star}, we get
\begin{equation*}
\ehk ((\bx), J; \omega_R) \leq  \ehk(I;\omega_R)
=\ehk(I) \le a\ehk(R).
\end{equation*}

In conclusion,
$$e = \ehk ((\bx), R) = \ehk (J, R) + \ehk((\bx), J) 
\leq f \ehk(R) + a\ehk(R) = (f+a) \ehk(R),$$
proving the first inequality stated in the conclusion.

The last inequality follows from the fact that $f = \length^*(R/J) \le
\length(R/J) = e -b$.
\end{proof}

The next corollary shows how useful Theorem~\ref{ehk-hom} can be when
$R$ is not Gorenstein.  Note that the lower bound for $\ehk(R)$ does not
depend on the dimension of the ring.

\begin{Cor}\label{Gor}  Let $(R,\fm)$ be a 
Cohen-Macaulay ring of CM-type $t$ and multiplicity
$\e = \e(R)$.  Then  $$\ehk(R) \ge
\dfrac{\e}{\e-t+1}.$$
\end{Cor}
\begin{proof}
By Remark~\ref{stdrmk}, we may assume that the residue field is infinite,
so there exists a  s.o.p. 
$\bx$ with  $\e(R) = \length(R/(\bx))$.
Now apply Theorem~\ref{ehk-hom} with $I = \fm$ (so $a=1$ and $b \geq t$).
\end{proof}
\begin{Cor}  Let $(R,\fm)$ be a non-regular, 
  Cohen-Macaulay ring of minimal multiplicity.
Then $\ehk(R) \ge \e(R)/2$.
\end{Cor}
\begin{proof}
By the structure theorem of Sally, \cite{Sa}, $R$ has type $t =\e(R)-1$.  Hence
$\ehk(R) \ge \e(R)/(\e(R) -(\e(R)-1)+1) = \e(R)/2$.
\end{proof}
\begin{Cor}\label{sfr}
Let
$(R,\fm,k)$ be a local Cohen-Macaulay  ring  of characteristic $p$ 
and dimension $d$.  If
$\ehk(R) < \dfrac{\e}{\e-1}$, then $R$ is Gorenstein and  F-regular
(so strongly F-regular, if $R$ is also excellent). 
\end{Cor}
\begin{proof}
We may assume that $R$ is not regular.  
If $R$ is not Gorenstein then the type of $R$, $t$, is at least $2$.
 Theorem~\ref{ehk-hom} then shows that
$\ehk \geq \dfrac{\e}{\e-t+1} \geq \dfrac{\e}{\e-1}$.
Thus $R$ is Gorenstein, and we are done by Theorem~\ref{be}.
\end{proof}

We can now state the desired generalization of Theorem~\ref{thm:BE}.
The improvement is replacing ``F-rational'' by an appropriate form
of ``F-regular'' in the conclusion. 
\begin{Cor}\label{smallehk} Let
$(R,\fm,k)$ be a formally unmixed ring  of characteristic $p$ and $\dim(R) = d \geq 2$.
  If
$\e_{HK}(R) \leq 1 + \op{max}\{1/d!,1/\e(R)\}$, then $R$ is 
F-regular and Gorenstein. 
If $R$ is excellent then $R$ is strongly F-regular. 
\end{Cor}

\begin{proof}  Let $\e =\e(R)$. We can pass to the completion and assume that $R$ is
  complete and unmixed. One should note that, for an excellent Gorenstein ring, strong F-regularity and F-regularity are equivalent.
Moreover if the completion of a ring $R$ is F-regular, then $R$ is F-regular.

Hence by Theorem~\ref{be} we
may assume that $R$ is Cohen-Macaulay.  

If $R$ is not strongly F-regular, then 
 $\ehk(R) \ge \e/(\e-1) > 1 + 1/\e$.  
So, $1 + \dfrac{1}{d!} \geq \ehk(R) > 1 + \dfrac{1}{\e} $ 
which implies that $\e > d!$, and therefore
 $\ehk(R) \ge  \dfrac{\e}{d!} > \dfrac{d!+1}{d!}$, which is a
 contradiction.

If $\e \geq d!+1$, then
since $\ehk(R) > \e /d!$ (this inequality is due to Hanes,~\cite{Ha}), 
we have $\ehk(R) > 1 + 1/d! > 1+1/\e$, a contradiction. Thus $\e \leq d!$,
so $\ehk(R) \leq 1 + 1/\e < \e/(\e-1)$, which implies that $R$ is
Gorenstein.

\end{proof}

It should be remarked that Corollaries 3.5 and 3.6 are closely related to recent unpublished results of D.~Hanes who
independently proved in particular that under the assumptions of Corollary 3.6, the ring $R$ is Gorenstein and F-regular.

We get some interesting results from Theorem~\ref{ehk-hom} when we can
apply it to Gorenstein rings which are not F-regular.
\begin{Cor}\label{Gor-ehk}
Let $(R,\m)$ be a Gorenstein ring of dimension $d$ and embedding dimension
$v = \mu(\m)$. Let $\e = \e(R)$. If either $R$ or $\hat R$ is not  F-regular, then
$$\ehk(R) \ge \dfrac {\e} {\e-v+d}.$$
\end{Cor}
\begin{proof}  Non F-regularity passes to the completion, so we may
assume that $R$ is complete.
By Remark~\ref{stdrmk}, we may assume that the
 residue field is infinite, and $\bx$ is s.o.p.~such that
$\e(R) = \length(R/(\bx))$, while preserving the
non-weak-F-regularity of $R$.   If $u$ denotes a socle element modulo $(\bx)$
then $u \in (\bx)^*$.  We can now apply Theorem~\ref{ehk-hom} with
$I  = \m$.  Then $a = \length^*(R/\m) = 1$, and 
$b = \length( ((\bx)^*:\m)/(\bx)) \ge v-d+1$, since in the $0$-dimensional 
Gorenstein ring $S = R/(\bx)$, $(u)S:\m S
= 0:\m^2 S$, and $\length( 0:\m^2 S) = \length(S/\m^2 S) = v-d+1$.
The corollary now follows.
\end{proof}
\begin{Rem}
It is possible, in ``pathological'' cases (e.g., non-excellent) for a
ring to be weakly F-regular, while its completion is not.  Loepp and
Rotthaus, construct such an example, which is Gorenstein, in \cite{LR}.
Corollary~\ref{Gor-ehk} applies in this case.
\end{Rem}
Corollary~\ref{Gor-ehk} can be improved, and this improvement, while interesting
on its own, will also be useful in section~\ref{root-compare}.  We first
establish some notation.  For a graded ring $G = \oplus_{i\ge0} G_i$, finitely
generated over $G_0$ artinian, let $k_i = \length(G_i)$.  If $\length(G) < \8$,
let $r = \max\{i | G_i \ne 0\}$.  We note that if $(S,\n)$ is a Gorenstein
ring of dimension $0$, and $G$ is the associated graded ring of $S$ at $\n$,
then $G_r$ is generated by the image of the socle element, so $k_r = 1$.

\begin{Cor}\label{Gor-ehk2}  Let $(R,\m)$ be a non F-regular
Gorenstein local ring of
dimension $d$ and multiplicity $\e = \e(R)$, 
and let $\bx = x_1,\ldots, x_d$ be a minimal reduction of 
$\m$.  Let $G$ be the associated graded ring of $R/(\bx)$ (at its maximal
ideal), and let $r$ and $k_i$ for $0\le i \le r$ be as above. Then
\[ \ehk(R) \ge \max_{1\le i \le r} \left\{\dfrac {\e}{\e-k_i}\right\}.
\]
As a consequence  
$\ehk(R) \ge 
\dfrac{\e}{\e-\frac{\e-2}{r-1}} \ge \dfrac{r+1}r$.
\end{Cor}
\begin{proof}
 Since $R$ is not F-regular, if $u$ denotes a 
socle element modulo $(\bx)$, then $u \in (\bx)^*$.  Thus
$(\bx):\m = (u,\bx)
\inc (\bx)^*$.  We may then apply Theorem~\ref{ehk-hom} with $I = \m^j +(\bx)$
and $J = (\bx)^*:I \supseteq (u,\bx):\m^j = \left((\bx):\m\right):\m^j
= (\bx):\m^{j+1}$.  In this case,
$\length(R/I) = \sum_{i=0}^{j-1} k_i$ and 
$ \length(R/J) = \e - \length(J/(\bx)) \le \e -   \length(R/(\m^{j+1} +(\bx)))=
 \e -(\sum_{i=0}^{j} k_i)$ (Matlis duality and the fact that
$J \subset (\bx):\m^{j+1}$ gives the inequality).  Hence 
\[\ehk(R) \ge \dfrac {\e}{\length^*(R/(\m^j+(\bx))) + \length^*(R/J)}
\ge \dfrac {\e}{\sum_{i=0}^{j-1} k_i + \e -(\sum_{i=0}^{j} k_i)} = \dfrac{\e}
{\e-k_j}.
\]
Since $k_0 = k_r = 1$,  some $k_i \ge \dfrac{\e-1-1}{r-1}$, thus 
$\ehk(R) \ge  \dfrac{\e}{\e - \frac{\e-2}{r-1}}$.  

Some algebra shows that $\dfrac{\e}{\e - \frac{\e-2}{r-1}} \ge \dfrac{r+1}r$
if and only if $\e \ge r+1$.  The latter condition always holds.
\end{proof}

\begin{Cor}\label{Gor-non-F-reg} 
Let $(R,\m)$ be a non F-regular Gorenstein ring of dimension
$d>1$.  Then $\ehk(R) \ge \dfrac {d+1}{d}$.  If $R$ is not a hypersurface,
then $\ehk(R) \ge \dfrac d{d-1}$.
\end{Cor}
\begin{proof}
By Remark~\ref{stdrmk} we may assume that $R$ is complete with
infinite residue field and  that $\bx$ is a s.o.p. which is a minimal reduction
of $\m$.

Let $G$ and $r$ be as in the proof
of Corollary~\ref{Gor-ehk2}.  The result of  Corollary~\ref{Gor-ehk2}
suffices if $r +1\le d$.  So we may assume that $r \ge d$. 
By the Brian\c con-Skoda
Theorem, $\m^d \inc (\bx)^*$.   

Let $\e = \e(R)$ be the multiplicity.  It is easy to see that  for any
integer  $n \le \e$, $\dfrac {\e}{\e -n} \ge \dfrac d{d-1}$ if and only
if $n \ge \e/d$.  By Corollary~\ref{Gor-ehk2}, we are done if some $k_i
\ge \e/d$,  so assume that each $k_i < \e/d$.

Let $I = \m^{d-1} + (\bx)$.  Then $(\bx)^*:I \supseteq \m$ (by the Brian\c con-Skoda
Theorem), so by Theorem~\ref{ehk-hom}, 
$\ehk(R) \ge \dfrac {\e}{\e-(\e-1)+1+k_1+\cdots + k_{d-2}}
= \dfrac{\e}{2 + k_1+\cdots + k_{d-2}}$.  Since each $k_i < \e/d$ we
get $\ehk(R) > \dfrac {\e}{2 + (d-2)(\e/d)}$, and the right hand side
is easily seen to be at least $\dfrac d{d-1}$ provided that $\e \ge 2d$.

The only case left is if $\e < 2d$.  
Then $2d > e > dk_i$ for all $k_i$ implies that each $k_i=1$, i.e., $R$
is a hypersurface, and $\e = r+1$ (and, recall, $r\ge d$).  
Say $\m = (z,\bx)$ minimally.  By the Brian\c con-Skoda theorem, $z^d\in
(\bx)^*$, so $(\bx)^*:\m \supseteq (z^d,\bx):z \supseteq (z^{d-1},\bx)$.
Applying Theorem~\ref{ehk-hom} with $I = \m$ gives
$\ehk(R) \ge \dfrac {\e}{1+ d-1} = \dfrac {\e}d \ge \dfrac{d+1}d$.
\end{proof}

\section{Radical extensions and comparison of Hilbert-Kunz multiplicities}
\label{root-compare}

In this section, we will develop a technique that, in conjuction with the
results  obtained so far, will give a lower bound for the Hilbert-Kunz
multiplicity of unmixed non-regular local rings of dimension $d$ that depends
only on $d$, and is strictly greater than $1$, hence showing that $\epsilon(d)
>0$. This answers one of the open questions mentioned in the Introduction.

We will need to use a result of Watanabe and Yoshida (\cite{WY} Theorem 2.7 and \cite{WY3} Theorem 1.6).  For a domain $R$ we use $Q(R)$ for the fraction
field of $R$, and $R^+$ for the absolute integral closure of $R$ (i.e.,
an integral closure of $R$ in an algebraic closure of $Q(R)$).

\begin{Thm}\label{extendehk}
Let $(R,\m) \into (S,\n)$ be a module-finite extension of local domains.
Then for every $\m$-primary ideal $I$ of $R$, 
\begin{equation}
\e_{HK}(I) = \dfrac{\e_{HK}(IS)}{[Q(S):Q(R)]} \cdot [S/\n:R/\m].
\end{equation}
\end{Thm}

We need the following definition.
\begin{Def} Let $(R,\m)$ be a domain. Let $z \in \m$,  and let $n$ 
be a positive integer.
Let $v \in R^+$ be any  root of $f(X)=X^n-z$.  We call $S = R[v]$ a {\it radical
extension} for the pair $R$ and $z$. 
\end{Def}

\begin{Rem}\label{normalradext}
 Whenever $S$ is radical for $R$ and $ z$, then
$b:=[Q(S):Q(R)] \leq n$.  Assume also that $R$ is normal and $z$ is a
minimal generator of $\m$.  Then in fact, $b = n$.  To see this we need
to show that $f(X) = X^n-z$ is the minimal polynomial for $v = z^{1/n}$ over
$R$.  Let $g(X)$ be the minimal polynomial of $v$ over $Q(R)$.  Since
$R$ is normal, $g(X) \in R[X]$.  The constant term of
$g(X)$ is in $\m$, since $z$ is not a unit.  Then $g(X)|f(X)$ in $R[X]$.  
Say $f(X) = g(X)h(X)$.  Then the constant term of $h(X)$ is a unit (or
else $z \in \m^2$).  But mod $\m$, $g(X)h(X) = X^n$, so in fact, $h(X)$
is a unit constant.
\end{Rem}

In what follows $\n$ will denote the maximal ideal of
$S$, whenever $S$ is local.  Note that if $R$ is a complete local domain and 
$z \in \m$, then $S$ must be local.

\begin{Thm}\label{useroots}
Let $(R,\m)$ be a complete local domain of positive prime characteristic
having algebraically closed residue field.  Let $\bx = x_1,\ldots, x_d$
be a system of parameters, and set $\e = \e_{HK}((\bx)) = \e((\bx))$,
and $a = \length(R/(\bx)^*)$.

Let $z \in \m - (\bx)^*$ be a minimal generator and  let $v \in R^+$ be any
$n$th root of $z$.  Let $S = R[v]$ be a radical extension for $R$ and $z$ and
denote the maximal ideal of $S$ by $\n$. Let $b = [Q(S):Q(R)]$. Then $$\ehk(R)
\geq \dfrac{b(n-1)\e +n \ehk(S)}{b(a(n-1)+1)}.$$  

In the case that $b=n$ this inequality simplifies to
$$\ehk(R) \ge \dfrac{(b-1)\e + \ehk(S)}{a(b-1)+1}.$$
\end{Thm}

\begin{Rem}If we denote $\e_{HK}(R) = 1 + \delta_R$ and $\e_{HK}(S) = 1 + \delta_S$,
then the above is equivalent to  
$$\delta_R \ge \dfrac{b(n-1)(\e-a) + n-b+n\delta_S}{b(a(n-1)+1)},$$
and if $b = n$ this simplifies to 
$\delta_R \ge \dfrac{(b-1)(\e-a)+\delta_S}{a(b-1)+1}$.
\end{Rem}
For the proof of Theorem~\ref{useroots} it is helpful to
note the following
\begin{Rem}\label{nestedehk}
Let $I \inc R$ be an ideal in a local ring $(R,\m)$ and $v\in\m$ an
element such that $(I,v)$ is $\m$-primary.  Then for all
$n \ge 1$, $\ehk((I,v^n),(I,v^{n-1})) \ge \ehk((I,v^{n+1}), (I,v^n))$.

To see this, we observe that for all $q$, $(I,v^n)\brq: v^{(n-1)q} \inc
(I,v^{n+1})\brq: v^{nq}$, so
\begin{align*}
\ehk((I,v^n),(I,v^{n-1})) &= 
\lim_{q\to\8} \dfrac 1{q^d}\length\left(\dfrac{(I,v^{n-1})\brq}{(I,v^{n})\brq}
\right)
= \lim_{q\to\8} \dfrac 1{q^d} \length\left(\dfrac R {(I,v^{n})\brq: v^{(n-1)q}}\right) \\
&\ge \lim_{q\to\8} \dfrac 1{q^d} \length\left(\dfrac R {(I,v^{n+1})\brq: v^{nq}}\right)
= \ehk((I,v^{n+1}), (I,v^n)).
\end{align*}
\end{Rem}

\begin{proof}  

Let $(\bx)^* =I_0 \subsetneq I_1 \subsetneq \cdots \subsetneq I_{a-2}
\subsetneq (I_{a-2},z)=I_{a-1} = \m \subsetneq R$ be a saturated filtration,
and let $w_i \in R$ be an element whose image generates $I_i/I_{i-1}$
(in particular, take $w_{a-1} = z$). 

We can then filter $(\bx)^*S \subseteq S$ by filling in each 
$I_{i-1}S \subseteq
I_{i}S$ with 
$$I_{i-1}S \subseteq (I_{i-1}, v^{n-1}w_i)S \subseteq \cdots
\subseteq (I_{i-1}, vw_i)S \subseteq I_i S$$ (where we allow that some of the containments
may be equalities).

From Theorem~\ref{extendehk}, and the fact that $[S/\n:R/\m] = 1$ ($R/\m$
is algebraically closed), we have that $\ehk(\m S) = b\ehk(\m R)$. 

Thus, $\ehk(\m S,\n) = b\ehk(R)-\ehk(S)$.

By Remark~\ref{nestedehk}, for each  $1\le j < n$,
$\ehk((v^j,\m S), (v^{j-1}, \m S) \ge \ehk((v^{j+1},\m S), (v^{j}, \m S)$.
Hence,  $\ehk((\m S), (v^{n-1}, \m S) \le 
\dfrac{\ehk(\m S,\n)}{n-1}$.

Set  $y : = \ehk((\m S), (v^{n-1}, \m S)$.
Consider the filtration
\begin{equation}\label{eqn:filtration}
\m S = (z,I_{a-2})S \supseteq (zv,I_{a-2})S \supseteq (zv^2,I_{a-2})S 
\supseteq \cdots
\supseteq (zv^{n-1},I_{a-2})S \supseteq I_{a-2}S.
\end{equation}

Remark~\ref{nestedehk} applies to each containment in 
equation~\ref{eqn:filtration},
so each relative Hilbert-Kunz
multiplicity is at most
$\ehk((zv,I_{a-2})S, \m S) = \ehk((v^{n+1}, I_{a-2})S, (v^n,I_{a-2})S)
\le y$. Adding
them all up we get that
$\ehk(I_{a-2}S, \m S) \le n y$.

From this it follows that $\ehk(I_{a-2}S, \m S) \le n \cdot \dfrac{\ehk(\m S,\n)}{n-1}$.  

Using Theorem~\ref{extendehk} to go back to $R$ we have 
$\ehk(I_{a-2}, \m ) \le n  \dfrac{\ehk(\m S,\n)}{b(n-1)} $.  Each of the other
$a-1$ terms in the filtration of $(\bx)^* \subseteq R$ have relative
Hilbert-Kunz multiplicity at most $\ehk(R)$, so we get the inequality
\begin{equation}\label{eqn:totalehk}
\left(n  \dfrac{\ehk(\m S,\n)}{b(n-1)}\right) + (a-1)\ehk(R)
\ge \ehk((\bx)^*) = \e.
\end{equation}

But   
equation~\ref{eqn:totalehk}
yields
 $$\ehk(R) \geq \dfrac{b(n-1)\e +n \ehk(S)}{b(a(n-1)+1)}.$$
\end{proof}

\begin{Cor}
\label{GorFreg}
Let $(R,\m)$ be an F-rational complete non-regular local ring of positive prime characteristic
having algebraically closed residue field.  Let $\bx = x_1,\ldots, x_d$
be a system of parameters and minimal reduction for $\m$, and let $\e = \e(R)= \e_{HK}((\bx)) = \e((\bx))$,

Let $z \in \m - (\bx)$ be a minimal generator and 
let $v \in R^+$ be any $n$th
root of $z$.  Let $S = R[v]$ be a radical extension for 
$R$ and $z$ and denote its maximal ideal of $S$ by $\n$.
 Then 
$$\ehk(R) \ge \dfrac{(n-1)\e + \ehk(S)}{\e(n-1)+1}.$$
\end{Cor}

\begin{proof} By Remark~\ref{normalradext},
$b = [Q(S):Q(R)] = n$.  Since $R$ is F-rational,
 $(\bx)= (\bx)^*$.  Hence one can apply
Theorem~\ref{useroots} together with the observation that $a =\e$.
\end{proof}
\begin{Rem}  Corollary~\ref{GorFreg} can be substantially improved,
but the proof is considerably more difficult.  We will give improved versions
in a later paper, along with improved estimates of lower bounds
for $\epsilon(d)$.
\end{Rem}
\begin{Cor}
Let $(R,\m)$ be a complete local domain of positive prime characteristic
having algebraically closed residue field.  Let $\bx = x_1,\ldots, x_d$
be a system of parameters and minimal reduction for $\fm$, and set $\e = \e_{HK}((\bx)) = \e((\bx))$,
and $a = \length(R/(\bx)^*)$.
 Then $$\ehk (R) \geq \dfrac{\e+1}{a+1}.$$
\end{Cor}
\begin{proof}
If $\m = (\bx)^*$ then $a =1$ and $\ehk(R) = \ehk((\bx)) = \e \ge (\e+1)/2$.

Otherwise, 
take any minimal generator of $\m$ not in $(\bx)^*$ and adjoin a square 
root of it from $R^+$. Then apply the
previous theorem and note that $2 = n \geq b$ and $\ehk(S) \geq 1$, so

$\ehk(R) \geq \dfrac{b(n-1)\e +b}{b(a(n-1)+1)} = 
\dfrac{(n-1)\e+1}{a(n-1)+1} = \dfrac{\e+1}{a+1}$.
\end{proof}

\begin{Rem}
Assume that $(R, \m)$ is CM of type $t$, $I$ a parameter ideal and minimal reduction for $\fm$
such that $I \subsetneq I^* \subsetneq \m$. 
Then $e = \length (R/I)$, and $t = \length ((I : \m )/I)$.

The two ideals $I^*$ and $(I :\m)$ are incomparable in many cases.

However, in the special case when $ (I: \m) \subseteq I^*$ (the Gorenstein case for example), then $t \leq e-a$ and 
$\dfrac{e}{e-t+1} \leq \dfrac{e+1}{a+1}$. So, the above corollary improves 
an earlier result of ours in this case.
\end{Rem}

\medskip

We now begin a construction that will yield a lower 
bound for the Hilbert-Kunz multiplicity of Gorenstein,
F-regular, non-regular local rings.  

So assume that $(R,\m)$ is a Gorenstein F-regular local ring of multiplicity
$\e = \e(R) > 1$.  Note that $R$ must be a normal domain.  We may complete and
by Theorem 3.4  of \cite{Ab}, extend the residue field to assume that it is
algebraically closed. 
Let $\bx = x_1,\ldots, x_d$ be a minimal reduction of $\m$, so that
$\length(R/(\bx)) = \e$.  

\begin{Rem}\label{keepreduction}
Let $R$ and $\bx$ be as above, and suppose $z$, $v= z^{1/n}$ and $S$ are as in
Corollary~\ref{GorFreg}.  Assume, moreover, that $x_1,\ldots, x_{d-1}, z$
is also a minimal reduction of $\m$.  Let $u \in \m$ denote a socle
element modulo $(x_1,\ldots,x_{d-1},z)$.  Then
\begin{itemize}
\item[a)] $x_1,\ldots, x_{d-1}, v$ is a minimal reduction of $\n$ (the
maximal ideal of $S$),
\item[b)] $u$ is still a socle element modulo $(x_1,\ldots,x_{d-1},v)S$, and
\item[c)] $S$ is Gorenstein and $\e(S) = \e(R)$.
\end{itemize}
\end{Rem}
\begin{proof}  Let $\bx_{d-1} = x_1,\ldots, x_{d-1}$.

a) If $\m = (\bx_{d-1},z) + J$, where $\mu(J) = \mu(\m)-d$, then
$\n = (\bx_{d-1},v)S + JS$.  Since $J$ is integral over
$(\bx_{d-1},z)R$, the ideal $JS$ is integral over
$(\bx_{d-1},z)S$, and hence over the larger ideal 
$(\bx_{d-1},v)S$.  This suffices to show (a).

b) If $u \in (\bx_{d-1}, v)S$ then $u \in (\bx_{d-1},z)S \cap R
\inc ((\bx_{d-1},z)R)^* = (\bx_{d-1},z)R$, a contradiction.  
With $J$ as in part (a), we have $\n\,u =((\bx_{d-1},v)S + JS)u
\inc J u S + (\bx_{d-1},v)S \inc (\bx_{d-1},z)S+ (\bx_{d-1},v)S
\inc (\bx_{d-1},v)S$.  Thus $u$ is a socle element.

c) By Remark~\ref{normalradext}, $X^n-z$ is the minimal polynomial of $v$ over
$R$.  Hence $S$ is $R$-free, so flat, with Gorenstein closed
fiber.  Thus $S$  is Gorenstein.
Then $\e(S) = \length_S(S/(\bx_{d-1}, v)) 
= \dfrac 1 n \length_S(S/(\bx_{d-1}, v^n)) 
=\dfrac 1 n\length_S(S/(\bx_{d-1}, z))  = \length_R(R/(\bx_{d-1},z))
= \e(R)$. 
% 
% Let $\e = \e(R)$ and let $R \supsetneq\m = I_1 \supsetneq
% \cdots \supsetneq I_{e-1} \supsetneq (x_1,\ldots, x_{d-1},z)R$ be
% a saturated filtration.  If  $I_j = I_{j+1} + (w_j)$ with
% $\m w_j \inc I_{j+1}$, then $\n w_j = (\m, v)S w_j \inc I_{j+1}S + vS$,
% so the chain
% $S \supseteq \n \supseteq I_2S+vS \supseteq \cdots
% \supseteq I_{e-1} \supseteq (x_1,\ldots, x_{d-1},v)S$
% and part (a) show that $\e(S) \le \length(S/(x_1,\ldots, x_{d-1},v)S)
% \le \e(R)$.
\end{proof}

Let $d = \dim R$ and $k = \mu(\m)-d > 1$.

Note that $\ehk(R) \geq \frac{\e(R)}{d!}$. Hence whenever $\e(R) \geq d!+1$, we
have that $\ehk(R) \geq 1 + \frac{1}{d!}$.

Therefore, if we want to produce a lower bound for $\ehk(R)$ in terms of only
$d$, there is no harm if we fix $\e(R)=\e$ as well. This is so because we can
take the minimum of the lower bounds obtained for fixed $d, \e$ while letting
$\e$ vary between  $2$ and $ d!$. 

The residue field of $R$ is infinite, and so we may pick $y_1,\ldots, y_{d+1}
\in \m - \m^2$ in general position, and therefore, assume that each
$d$-element subset is a minimal reduction of $\m$ (see, for example, 
Theorem 8.6.6 of \cite{HS}, and the comment after it). 
Let $u$ denote a socle element modulo $(y_1,\ldots, y_d)R$, and let 
$r = \max\{\, i \mid u \in \m^i + (y_1,\ldots, y_d)R\}$.  
Set $n = \lceil d/r \rceil$ (so $nr \ge d$).

 Let $R_0 = R$, and
for each $i\ge 1$, let $v_i = y_i^{1/n}$, and set $R_i = R_{i-1}[v_i]$.
For each $i$, write $\ehk(R_i) = 1 + \delta_i$.  

For a given $i\ge1$, if $R_{i-1}$ is F-regular, we may apply 
Corollary~\ref{GorFreg} to $R_{i-1} \inc R_i$ with $\bx
= v_1,\ldots, v_{i-1}, y_{i+1},\ldots, y_{d+1}$ and $z = y_i$
($\bx$ is a minimal reduction of $R_{i-1}$ by Remark~\ref{keepreduction}(a)).
Also, by
Remark~\ref{keepreduction}(b),  $u$ is a socle element modulo 
$(v_1,\ldots, v_i,y_{i+1},\ldots, y_d)R_i$.
We get, noting that the multiplicity stays the same,
\begin{equation}\label{eqn:compareehk}
1 + \delta_{i-1} \ge 1 + \dfrac{1}{\e(R_{i-1})(n-1)+1} \delta_{i}
= 1 + \dfrac 1 {\e(R_0)(n-1)+1} \delta_{i}.
\end{equation}
We claim that for some $i\le d$, $R_i$ is not F-regular.  If not,
then $R_d$ is F-regular.  Let $\m_{R_0} = (y_1,\dots, y_d) + J$
with $\mu(J) = \mu(\m)-d$.  It is then clear that $\m_{R_d} = (v_1,
\ldots, v_d) + J R_d$.
By the Brian\c con-Skoda Theorem $\overline{\m_{R_d}^d} \subseteq
\left((v_1,\ldots, v_d)R_d\right)^*$, so
\begin{align*} u \in  (JR_0)^r
 &\inc \overline{(y_1,\ldots, y_d)^rR_d} = \overline{(y_1^r,\ldots, y_d^r)R_d}
= \overline{(v_1^{rn},\ldots, v_d^{rn})R_d}
\\
&\inc\overline{(v_1^d,\ldots, v_d^d)R_d} 
\inc \left((v_1,\ldots, v_d)R_d\right)^*
= (v_1,\ldots, v_d)R_d
\end{align*} 
a contradiction to 
Remark~\ref{keepreduction}(b).

Assume then, that $i_0 = \min\{i \mid R_i \text{ is not F-regular}\}$.
By Corollary~\ref{Gor-non-F-reg}, 
$\ehk(R_i) \ge \dfrac {d+1}{d} = 1 + \dfrac 1{d}$.
Repeated application of Equation~\ref{eqn:compareehk} yields
\begin{equation*}
\ehk(R) = \ehk(R_0) \ge 1 + \left(\dfrac 1{\e(R)(n-1)+1}\right)^{i_0}
\dfrac 1 {d}.
\end{equation*}

We are now in position to state and prove the main result of the paper.
\begin{Thm}\label{lowerbound} 
Let $\ringR$ be a formally unmixed local ring of positive characteristic $p$
and dimension $d\geq2$.
If $R$ is not regular then 
$$\ehk(R) \geq 1+ \frac{1}{d \cdot (d! (d-1)+1)^{d}}.$$
\end{Thm}

\begin{proof} We can make a faithfully flat extension so we can assume that $k$ is
algebraically closed and that $R$ is also complete.

We can assume that $\ehk(R) <1 +\frac{1}{d!}$ and
hence by Corollary~\ref{smallehk} we have that $R$ is Gorenstein and
F-rational, hence strongly F-regular.

If $\e \geq d!+1$, then $\ehk(R) \geq \frac{\e(R)}{d!} \geq 1 +\frac{1}{d!}.$
So, we can assume that $\e \leq d!$.

Now we are in position to apply the technique described just above the statement
of the Theorem and, noting that $n \le d$ we obtain that

$$\ehk(R) \geq 1+\frac{1}{(d \cdot(\e(d-1)+1)^{d}} 
\geq 1+ \frac{1}{d \cdot (d!(d-1)+1)^{d}}.$$
\end{proof}

\end{document}